\documentclass[12pt,oneside]{amsproc}

\usepackage[cp1251]{inputenc}
\usepackage[russian]{babel}
\usepackage{amssymb,epic,eepic}
\usepackage[dvips]{graphicx}
\numberwithin{equation}{section}
\newtheorem{lem}{Лемма}[section]
\newtheorem{thm}{Теорема}[section]

\newtheorem{rim}{Замечание}[section]

\newtheorem{hyp}{Гипотеза}[section]

\newcommand{\Wo}{{\raisebox{0.2ex}{\(\stackrel{\circ}{W}\)}}{}}

\title[]{Явный вид  экстремалей в задаче о  константах вложения в пространствах Соболева.}
\author{Т.А.Гарманова, И.А.Шейпак
\address{Московский государственный университет
им.~М.~В.~Ломоносова, механико-математический факультет}
\email{garmanovata@gmail.com, iasheip@yandex.ru}
}
\thanks{Результаты разделов 2 и 3 получены при поддержке гранта РФФИ № 19-01-00240, результаты разделов 4--6 получены при  поддержке гранта РНФ №~17-11-01215.}

\begin{document}
\noindent УДК~517.518.23, 517.984
\begin{abstract}
В статье рассматриваются константы вложения соболевских пространств $\Wo^n_2[0;1]\hookrightarrow
\Wo^k_2[0;1]$ ($0\leqslant k\leqslant n-1$). Указана связь констант вложения с нормами функционалов $f\mapsto f^{(k)}(a)$ в пространстве $\Wo^n_2[0;1]$. Найден явный вид функций $g_{n,k}\in\Wo^n_2[0;1]$, на которых рассматриваемые  функционалы достигают своей нормы. Эти же функции являются экстремальными для констант вложения. Исследуется связь констант вложения с полиномами Лежандра. Подробно изучены свойства констант вложения при $k=3$ и $k=5$: получены явные формулы для точек экстремума, определены точки глобального максимума и найдены значения точных констант вложения. Установлена связь между константами вложения и некоторым классом спектральных задач с коэффициентами-распределениями.
\end{abstract}
\maketitle

{\small
\textbf{Ключевые слова: }\textit{Пространства Соболева, константы вложения, полиномы Лежандра}

\textbf{Key words: }\textit{Sobolev spaces, embedding constants, Legendre polinomials}

\section{Введение}
Мы будем рассматривать пространство Соболева $\Wo_2^n[0;1]$, состоящее  из всех функций $f$, все производные которых до порядка $n-1$ абсолютно непрерывны на отрезке $[0;1]$, и $f^{(n)}\in L_2[0;1]$, а также выполнены краевые условия
$$
f^{(j)}(0)=f^{(j)}(1)=0,\quad j=0,1,\ldots, n-1.
$$

Введем обозначение $\mathcal H:=\Wo_2^n[0;1]$. Норма в пространстве  $\mathcal H$ определяется равенством
$$
\|f\|_{\mathcal H}:=\left(\int_0^1\left|f^{(n)}(x)\right|^2dx\right)^{1/2}.
$$

Для фиксированной точки $a\in(0;1)$ и целого числа $k\in\{0,1,\ldots, n-1\}$ рассмотрим задачу о нахождении явных выражений для величин

\begin{equation}\label{eq:ekstr}
A_{n,k}(a):=\sup\left\{|f^{(k)}(a)|:\quad \|f\|_{\mathcal H}\leqslant 1\right\}.
\end{equation}

Другими словами, числа $A^2_{n,k}(a)$ представляют собой наименьшие возможные константы в неравенствах
\begin{equation}\label{eq:MFK}
\left|f^{(k)}(a)\right|^2\leqslant A^2_{n,k}(a)\|f\|^2_{\mathcal H}, \quad f\in\mathcal H.
\end{equation}

Величина
\begin{equation}\label{eq:const}
\Lambda_{n,k}:=\sup_{a\in[0;1]} A_{n,k}(a)
\end{equation}
равна норме оператора вложения $\mathcal H\to \Wo^k_\infty[0;1]$. Число $\Lambda_{n,k}$ также называют \textit{точной константой вложения.}

Задачи об оценках производных функций в различных функциональных пространствах привлекали интерес многих математиков. А.\,А.\,Марковым была решена задача о нахождении наибольшего значения производной многочлена на отрезке $[-1;1]$ среди всех многочленов, ограниченных по модулю единицей. Неравенства для производных в пространствах $L_p(\mathbb{R})$ или $L_p(\mathbb{R}^+)$ называют неравенствами колмогоровского типа. Эти неравенства естественно возникают в теории приближений, в частности в задаче Стечкина о приближении неограниченных операторов. В работе \cite{Kalyab} неравенства вида \eqref{eq:MFK} названы неравенствами типа Маркова--Фридрихса--Колмогорова.

Вопросам вычисления точных констант в неравенствах колмогоровского типа для промежуточных производных в различных случаях посвящены многие работы. См., например, монографии \cite{Tikh}, \cite{MagTikh}, в которых дан исторический обзор и приведена обширная библиография по данной тематике.

В работе \cite{Kalyab} получена рекуррентная формула, связывающая величину  $A^2_{n,k}(a)$ с величиной $A^2_{n-k,0}(a)$ и значениями специально подобранных первообразных полиномов Лежандра в точке $a$. С помощью этой формулы получены явные выражения для $\Lambda^2_{n,0}$,  $\Lambda^2_{n,1}$ и $\Lambda^2_{n,2}$.

В работе \cite{NazM} вычислены константы $\Lambda^2_{n,4}$ и $\Lambda^2_{n,6}$. Заметим, что при больших $k$ и $n$ трудности, связанные с  вычислением величин $A^2_{n,k}$ значительно возрастают. При этом задача о нахождении точных констант вложения при нечетных значениях $k$ существенно сложнее, чем при четных $k$. Изучение величин $\Lambda^2_{n,k}$ и других констант вложения важно не только с точки зрения теории аппроксимации в  различных пространств Соболева, но и представляет интерес для спектральной теории операторов.

В \cite{BotWid}   исследована связь норм операторов вложения $\Wo^n_2[-1;1]\to \Wo^k_2[-1;1]$ c асимптотическим поведением первого (наименьшего) собственного значения $\lambda_{1,n}$ при $n\to\infty$ задачи
$$
(-1)^ny^{(2n)}=\lambda y, \quad y^{(j)}(-1)=y^{(j)}(1)=0,\quad j=0,1,\ldots,n-1
$$
и для $\lambda_{1,n}$ получены двусторонние оценки, отношение которых стремится к бесконечности.
В работе \cite{Kalyab} с помощью  явного вида $\Lambda^2_{n,0}$ получены существенно более точные по сравнению с результатами работы \cite{BotWid} двусторонние \textit{сближающиеся} оценки для $\lambda_{1,n}$ при $n\to \infty$.

В \cite{Kalyab} выдвинута гипотеза, что при четных $k$ максимум $A^2_{n,k}$ достигается в середине отрезка, а при нечетных $A^2_{n,k}$ имеет в середине отрезка локальный минимум. Также в этой работе ставится вопрос о вычислении точных значений $\Lambda^2_{n,k}$ при $k\geqslant 3$ (частично решенный в \cite{NazM}). \textit{Экстремалью} или \textit{экстремальной} называется функция, для которой в неравенствах \eqref{eq:MFK} достигается равенство. В работе \cite{Kalyab}  указывается, что на данный момент  не имеется конструктивного описания экстремальных сплайнов, реализующих максимальное значение $|f^{(k)}(a)|$ при $f\in \mathcal H$ даже при $k=0$. В работе \cite{NazM} подробно исследуются свойства экстремалей задач вложения  и показано, что при нечетных $k<n$ экстремаль не обладает симметрией относительно середины отрезка, а при всех четных $k<n$ четная относительно середины отрезка функция дает локальный максимум задачи \eqref{eq:ekstr}. Также в \cite{NazM} предъявлены явные формулы для $\Lambda^2_{n,4}$ и $\Lambda^2_{n,6}$, но они не верны. Причиной ошибки, скорее всего, является возрастающая с увеличением $n$ и $k$ сложность вычислений.

Данная работа преследует несколько целей. 1) Получить явные формулы экстремальных сплайнов, реализующих максимальное значение $|f^{(k)}(a)|$ на классе $\mathcal H$. 2) Получить новую рекуррентную формулу для функций $A^2_{n,k}$, позволяющую вычислять  константы вложения существенно более просто, нежели позволяют методы работ \cite{Kalyab} и \cite{NazM}. 3) Получить явные формулы для $A^2_{n,k}(a)$ при $k=3$, $k=5$ и найти их глобальный максимум по переменной $a$.
4) Указать связь величин $A^2_{n,k}(a)$ с другим, нежели рассмотренным в работах \cite{Kalyab} и \cite{BotWid}, классом спектральных задач.

Отметим, что результаты работ \cite{Kalyab} и \cite{NazM} относятся к вычислению точных констант на отрезке $[-1;1]$. Нам более удобно провести аналогичные исследования на отрезке $[0;1]$. Это связано во-первых, с некоторыми спектральным задачами, которые традиционно рассматривают на отрезке $[0;1]$, а во-вторых, позволит не учитывать в вычислениях степени двойки, возникающие в нормировке стандартных полиномов Лежандра.  В \S \ref{sec:BP} мы укажем связь констант вложения с некоторым классом граничных задач с весом-распределением и приведем простой метод, позволяющий пересчитывать значения $A^2_{n,k,[0;1]}$  на отрезке $[0;1]$, в $A^2_{n,k,[-1;1]}$ --- функций, определенных на отрезке  $[-1;1]$.

Структура работы следующая. В \S 2 указывается связь между величинами $A^2_{n,k}(a)$ и нормами функционалов $f\mapsto f^{(k)}(a)$ в пространстве $\Wo^{n}_2[0;1]$ ($0\leqslant k\leqslant n-1$, $a\in[0;1]$), а также доказывается теорема о явном виде сплайнов, на котором указанные функционалы достигают своей нормы. В третьем разделе получены некоторые вспомогательные свойства первообразных полиномов Лежандра и функций  $A^2_{n,k}(a)$. В \S 4 и 5 получены точные значения констант вложения $\Lambda^2_{n,3}$ и $\Lambda^2_{n,5}$ соответственно. В \S 6 приведен класс спектральных задач тесно связанных с константами вложения, а также указан способ пересчета констант вложения, определенных на отрезке $[0;1]$, в константы вложения, определенные на отрезке $[-1;1]$.

На протяжении всей статьи через $\mathcal H$ мы обозначаем пространство  $\Wo^n_2[0;1]$, т.е. пространство функций $y\in W^n_2[0;1]$, удовлетворяющих краевым условиям  $y^{j}(0)=y^{j}(1)=0$, $j=0,1,\ldots,n-1$.

Через $P_n$ мы обозначаем смещенные полиномы Лежандра, образующие ортогональный базис в пространстве $L_2[0;1]$, и определяемые формулой Родрига
$$
P_n(x):=\dfrac{1}{n!}\bigl((x^2-x)^n\bigr)^{(n)}.
$$
Также нам понадобятся первообразные полиномов Лежандра. Первообразная порядка $m\geqslant0$ многочлена $P_n$ определяется формулой
$$
\left(P_n(x)\right)^{(-m)}:=\dfrac{1}{n!}\bigl((x^2-x)^n\bigr)^{(n-m)}.
$$
Многие полезные свойства первообразных полиномов Лежандра на отрезке $[-1;1]$ исследованы в работе \cite{HolSh}.

Термин \textit{точная константа вложения} обозначает число \eqref{eq:const} или его квадрат.

\section{Вид экстремальных сплайнов.}

Для функций $f\in \mathcal H$, некоторого фиксированного числа $a\in(0;1)$ рассмотрим функционалы $F_{k,a}(f)=f^{(k)}(a)$, $k=0,1,\ldots,n-1$.

Очевидна справедливость соотношения
\begin{equation}\label{eq:norm}
\|F_{k,a}\|^2=A^2_{n,k}(a).
\end{equation}

Поскольку указанные функционалы непрерывны в $\mathcal H$, то в соответствии с теоремой Рисса существует  единственная функция $g_{n,k}\in \mathcal H$ такая, что $F_{k,a}(f)=(f,g_{n,k})_{\mathcal H}$.

При этом очевидна справедливость соотношения
\begin{equation}\label{eq:norm}
\|F_{k,a}\|^2=\|g_{n,k}\|^2_{\mathcal H}=A^2_{n,k}(a).
\end{equation}

Поэтому
$$
f^{(k)}(a)=\int_0^1 f^{(n)}(x)\overline{g^{(n)}_{n,k}(x)}\,dx.
$$

Разбивая интеграл в сумму
$$
\int_0^1 f^{(n)}(x)\overline{g_{n,k}^{(n)}(x)}\,dx=\int_0^a f^{(n)}(x)\overline{g_{n,k}^{(n)}(x)}\,dx+\int_a^1 f^{(n)}(x)\overline{g_{n,k}^{(n)}(x)}\,dx
$$
и интегрируя каждый интеграл по частям, получаем
\begin{multline*}
f^{(k)}(a)=f^{(n-1)}(a)\left(\overline{g_{n,k}^{(n)}(a-0)}-\overline{g_{n,k}^{(n)}(a+0)}\right)-f^{(n-2)}(a)\left(\overline{g_{n,k}^{(n+1)}(a-0)}-\overline{g_{n,k}^{(n+1)}(a+0)}\right)+
\ldots+\\
+(-1)^{n-k-1}f^{(k)}(a)\left(\overline{g_{n,k}^{(2n-k-1)}(a-0)}-\overline{g_{n,k}^{(2n-k-1)}(a+0)}\right)+\ldots+
(-1)^{n-1}f(a)\left(\overline{g_{n,k}^{(2n-1)}(a-0)}-\overline{g_{n,k}^{(2n-1)}(a+0)}\right)+
\\+(-1)^{n}\int f(x)\overline{g_{n,k}^{(2n)}(x)}\,dx.
\end{multline*}

Отсюда получаем следующие условия на функцию $g_{n,k}$:
\begin{gather}\label{eq:uslovia_g}
\left.g_{n,k}^{(2n)}\right|_{[0;a)}=\left.g_{n,k}^{(2n)}\right|_{(a;1]}\equiv0,\\
g_{n,k}^{(i)}(a-0)=g_{n,k}^{(i)}(a+0),\quad i\ne 2n-k-1,\\
g_{n,k}^{(i)}(a-0)=g^{(i)}(a+0)+(-1)^{n-k-1}, \quad i=2n-k-1
\end{gather}

Рассмотрим следующие многочлены:

$$
h_{n,k}(x,a)=\sum_{l=0}^{n-1}(-1)^{n-1-l}C_{2n-1-k}^{n-1-l}x^{n-1-l}a^l\sum_{m=0}^{l}C_{n-1+m}^{m}x^{m}.
$$

\begin{thm}\label{thm:splain}
Функции $g_{n,k}$ определяются формулами:
\begin{equation}\label{eq:g}
g_{n,k}(x)=\left\{\begin{aligned}
&\dfrac{(-1)^{n-k-1}}{(2n-k-1)!}(1-a)^{n-k}x^n h_{n,k}(1-x,1-a),\quad x\in [0;a]\\
&\dfrac{(-1)^{n-1}}{(2n-k-1)!}a^{n-k}(1-x)^{n} h_{n,k}(x,a), \quad x\in [a;1].
\end{aligned}\right.
\end{equation}
\end{thm}

\begin{proof}
1) Очевидно что на каждом из отрезков $[0;a]$ и $[a;1]$ $g_{n,k}$ есть многочлен степени $2n-1$.

2) Также очевидно что $g_{n,k}$ удовлетворяет краевым условиям.

3) Рассмотрим разность $g_{n,k}(a+0)-g_{n,k}(a-0)$.

\textbf{1.}
Обозначим
\begin{gather*}
g_1(x):=(-1)^{n-k-1}(2n-k-1)! g_{n,k}|_{[0,a]},\\
g_2(x):= (-1)^{n-1}(2n-k-1)! g_{n,k}|_{[a,1]}.
\end{gather*}
Покажем, что $g_1(x)-(-1)^k g_2(x) = (x-a)^{2n-k-1}.$
Перепишем $h_{n,k}$ в другом виде:
\begin{multline*}
h_{n,k}(x,a)=\sum_{l=0}^{n-1}(-1)^{n-1-l}C_{2n-1-k}^{n-1-l}x^{n-1-l}a^l\sum_{m=0}^{l}C_{n-1+m}^{m}x^{m}=\\=
\sum_{j=0}^{n-1}\sum_{m=0}^{n-1-j}(-1)^{j}x^{j+m}a^{n-1-j}C_{2n-1-k}^{j}C_{n-1+m}^{m}=\sum_{j=0}^{n-1}
\sum_{i=j}^{n-1}(-1)^{j}x^{i}a^{n-1-j}C_{2n-1-k}^{j}C_{n-1+i-j}^{i-j}=\\
=\sum_{i=0}^{n-1}x^{i}\sum_{j=0}^{i}(-1)^{j}a^{n-1-j}C_{2n-1-k}^{j}C_{n-1+i-j}^{i-j}.
\end{multline*}

Рассмотрим
$$
g_2(x)=a^{n-k}(1-x)^{n} h_{n,k}(x,a) = \left(\sum\limits_{m=0}^n (-1)^m C_n^m x^m\right)\sum_{i=0}^{n-1}x^{i}\sum_{j=0}^{i}(-1)^{j}a^{2n-1-k-j}C_{2n-1-k}^{j}C_{n-1+i-j}^{i-j}.
$$

Раскрыв суммы в последнем равенстве, получаем, что коэффициент при $x^m,$ $m=0,\ldots,n-1$ в $g_2$ равен
 \begin{multline}\label{mu:xm}
 \sum_{i=0}^{m} (-1)^{m-i} C_n^{m-i} \sum_{j=0}^{i} (-1)^j C_{2n-1-k}^j C_{n-1+i-j}^{i-j} a^{2n-1-k-j}=\\=
 \sum_{j=0}^{m} (-1)^{j} a^{2n-1-k-j} C_{2n-1-k}^j \sum_{i=j}^{m} (-1)^{m-i} C_n^{m-i} C_{n-1+i-j}^{i-j}
 \end{multline}
 Тогда коэффициент при $x^m a^{2n-1-k-j},$ $ j=0,\ldots,m-1$ в $g_2$ равен
 \begin{multline*}
(-1)^{j} C_{2n-1-k}^j\sum_{i=j}^{m} (-1)^{m-i} C_n^{m-i} C_{n-1+i-j}^{i-j}= \sum_{u=0}^{m-j} (-1)^{u} C_n^{u} C_{n-1+m-j-u}^{m-j-u}=\\=
 \sum_{u=0}^{v} (-1)^{u} C_n^{u} C_{n-1+v-u}^{v-u} = 0.
 \end{multline*}
 Таким образом, коэффициент при $x^m a^{2n-1-k-j}$ в $g_2$ отличен от нуля только при $j=m$, следовательно,
 коэффициент при $x^m$ в функции $g_2$ равен $(-1)^{m} a^{2n-1-k-m} C_{2n-1-k}^m$.

\textbf{2.} Теперь рассмотрим $g_1(x)$. После преобразований
\begin{multline*}
(1-a)^{n-k}x^n h_{n,k}(1-x,1-a) = (1-a)^{n-k}x^n \sum_{i=0}^{n-1}(1-a)^i\sum_{j=0}^{i}(-1)^{n-1-i}C_{2n-1-k}^{n-1-i}(1-x)^{n-1-i+j}C_{n-1+j}^{j}=\\
=\left(\sum_{m=0}^{n-k} (-1)^m C_{n-k}^m a^m \right) x^n \sum_{i=0}^{n-1}\sum_{l=0}^i C_i^l (-1)^l a^l \sum_{j=0}^{i}(-1)^{n-1-i}C_{2n-1-k}^{n-1-i}(1-x)^{n-1-i+j}C_{n-1+j}^{j}=\\
=\left(\sum_{m=0}^{n-k} (-1)^m C_{n-k}^m a^m \right) x^n \sum_{l=0}^{n-1} a^l \sum_{i=l}^{n-1}\sum_{j=0}^{i}  C_i^l (-1)^l  (-1)^{n-1-i}C_{2n-1-k}^{n-1-i}(1-x)^{n-1-i+j}C_{n-1+j}^{j}
\end{multline*}
получаем, что коэффициент при $a^m,$ $m=0,\ldots,n-k-1$ в $g_1$ равен
\begin{multline*}
x^n \sum_{l=0}^{m} (-1)^{m-l} C_{n-k}^{m-l} \sum_{i=l}^{n-1}\sum_{j=0}^{i}  C_i^l (-1)^l  (-1)^{n-1-i}C_{2n-1-k}^{n-1-i}(1-x)^{n-1-i+j}C_{n-1+j}^{j}=\\
= (-1)^m x^n \sum_{l=0}^{m}\sum_{i=l}^{n-1}\sum_{j=0}^{i} C_{n-k}^{m-l} C_i^l(-1)^{n-1-i}C_{2n-1-k}^{n-1-i}(1-x)^{n-1-i+j}C_{n-1+j}^{j}=\\
= (-1)^m x^n \sum_{i=0}^{n-1}\sum_{l=0}^{min(i,m)}C_{n-k}^{m-l} C_i^l\sum_{j=0}^{i} (-1)^{n-1-i}C_{2n-1-k}^{n-1-i}(1-x)^{n-1-i+j}C_{n-1+j}^{j}.
\end{multline*}
Воспользуемся формулой (свертка Вандермонда): $\sum\limits_lC_{n-k}^{m-l} C_i^l = C_{n+i-k}^m$. Тогда коэффициент при $a^m$ в $g_1$ равен
$$
(-1)^m x^n \sum_{i=0}^{n-1}\sum_{j=0}^{i} C_{n+i-k}^m (-1)^{n-1-i}C_{2n-1-k}^{n-1-i}(1-x)^{n-1-i+j}C_{n-1+j}^{j}.
$$
Сделаем замены $v=n-1-i$ и $u=v+j$. Тогда коэффициент при $a^m$ в $g_1$ примет вид
\begin{multline*}
(-1)^m x^n \sum_{v=0}^{n-1}\sum_{j=0}^{n-1-v} C_{2n-1-k-v}^m (-1)^{v}C_{2n-1-k}^{v}(1-x)^{v+j}C_{n-1+j}^{j}=\\
=(-1)^m x^n
\sum_{v=0}^{n-1}\sum_{u=v}^{n-1} C_{2n-1-k-v}^m (-1)^{v}C_{2n-1-k}^{v}(1-x)^{u}C_{n-1+u-v}^{u-v}=\\
=(-1)^m x^n \sum_{u=0}^{n-1}(1-x)^{u} \sum_{v=0}^{u} C_{2n-1-k-v}^m (-1)^{v}C_{2n-1-k}^{v} C_{n-1+u-v}^{u-v}.
\end{multline*}

Заметим, что $C_{2n-1-k-v}^m  C_{2n-1-k}^{v} = C_{2n-1-k}^{m} C_{2n-1-k-m}^{v}$. С учетом этого равенства коэффициент при $a^m$ примет вид
\begin{multline*}
(-1)^m C_{2n-1-k}^{m} x^n \sum_{u=0}^{n-1}(1-x)^{u} \sum_{v=0}^{u}  (-1)^{v} C_{2n-1-k-m}^{v} C_{n-1+u-v}^{u-v}=\\
=(-1)^m C_{2n-1-k}^{m} x^n \sum_{u=0}^{n-1}(1-x)^{u} \sum_{v=0}^{u} (-1)^{v}C_{2n-1-k-m}^{v} C_{n-1+u-v}^{u-v}=\\
=(-1)^m C_{2n-1-k}^{m} \sum_{u=0}^{n-1} \sum_{p=0}^{u} (-1)^p C_u^p x^{n+p} \sum_{v=0}^{u} (-1)^{v}C_{2n-1-k-m}^{v} C_{n-1+u-v}^{u-v}=\\
=(-1)^m C_{2n-1-k}^{m} \sum_{p=0}^{n-1} x^{n+p} \sum_{u=p}^{n-1} (-1)^p C_u^p \sum_{v=0}^{u} (-1)^{v}C_{2n-1-k-m}^{v} C_{n-1+u-v}^{u-v}.
\end{multline*}
Тогда коэффициент при $a^m x^{n+p}$ равен
\begin{equation}\label{eq:xnp}
(-1)^m C_{2n-1-k}^{m} \sum_{u=p}^{n-1} (-1)^p C_u^p \sum_{v=0}^{u} (-1)^{v}C_{2n-1-k-m}^{v} C_{n-1+u-v}^{u-v}
\end{equation}
Так как
$$
 \sum_{v=0}^{u}  (-1)^{v} C_{2n-1-k-m}^{v} C_{n-1+u-v}^{u-v} = (-1)^u C_{n-k-m-1}^u,
 $$
то при $p=n-1-k-m$ получаем
$$
\sum_{u=p}^{n-1} (-1)^{p+u} C_u^{p} C_{n-k-m-1}^u = 1
$$
При $p>n-1-k-m$ один из множителей $C_u^{p}$ или $ C_{n-k-m-1}^u$ всегда равен 0, и, следовательно, равна нулю и вся сумма.
При $p<n-1-k-m$ воспользуемся тождеством
$$
C_u^p C_{n-k-m-1}^u = C_{n-k-m-1}^p C_{n-k-1-m-p}^{u-p}
$$
тогда,
$$
\sum_{u=p}^{n-1} (-1)^{p+u} C_u^{p} C_{n-k-m-1}^u = C_{n-k-m-1}^p \sum_{u=p}^{n-1-k-m} (-1)^{p+u} C_{n-k-1-m-p}^{u-p}=
$$
$$
=C_{n-k-m-1}^p \sum_{l=0}^{n-1-k-m-p} (-1)^{l} C_{n-k-1-m-p}^{l}=0
$$
Таким образом, коэффициент при $a^m x^{2n-1-k-m}$ равен $(-1)^m C_{2n-1-k}^{m},$ и коэффициенты при остальных степенях равны $0.$

Рассмотрим коэффициенты при оставшихся степенях $x^m a^l,$ $m=n,...,2n-1$ $l=n-k,...,2n-1-k$ в $g_1(x)$ и $g_2(x)$.

Действуя аналогично выводу формулы \eqref{mu:xm}, получаем, что коэффициент при $x^{m}a^{2n-1-k-j}$ в функции $g_2$ равен:
\begin{equation}\label{eq:g2}
(-1)^{j} C_{2n-1-k}^j \sum_{i=j}^{n-1} (-1)^{m-i} C_n^{m-i} C_{n-1+i-j}^{i-j}=
(-1)^{j+m-n} C_{2n-1-k}^j \frac{m-2n}{m-j}C_{2n-j-1}^n C_n^{m-n}.
\end{equation}
Здесь мы воспользовались тождеством
$$
\sum_{i=j}^{m} (-1)^{m-i} C_n^{m-i} C_{n-1+i-j}^{i-j}=(-1)^{m-n} \frac{m-2n}{m-j}C_{2n-j-1}^n C_n^{m-n}.
$$

Действуя аналогично выводу формулы \eqref{eq:xnp}, получаем, что
коэффициент при $a^{m}x^{n+p}$ в функции $g_1$ равен:
$$
(-1)^m C_{2n-1-k}^{m} \sum_{u=p}^{n-1} (-1)^p C_u^p \sum_{v=0}^{u} (-1)^{v}C_{2n-1-k-m}^{v} C_{n-1+u-v}^{u-v}.
$$
Тогда коэффициент при $x^ma^{2n-1-k-j}$ равен
$$
(-1)^{2n-1-k-j}C_{2n-k-1}^j\sum_{u=m-n}(-1)^{m-n}C_u^{m-n}\sum_{v=0}^u(-1)^v C_j^v C_{n-1+u-v}^{u-v}.
$$

Преобразуем  внутреннюю сумму:
$$
\sum_{v=0}^u(-1)^v C_j^v C_{n-1+u-v}^{u-v}=C_{n+u-j-1}^{n-j-1}.
$$

Продолжая преобразование полученной суммы с внешней суммой, получаем
\begin{multline}\label{eq:g1}
(-1)^{2n-1-k-j}C_{2n-k-1}^j\sum_{u=m-n}^{n-1}(-1)^{m-n}C_u^{m-n}C_{n+u-j-1}^{n-j-1}=\\=
(-1)^{2n-1-k-j}C_{2n-k-1}^j(-1)^{m-n+1}\frac{m-2n}{m-j}C_n^{m-n}C_{2n-j-1}^n.
\end{multline}

Таким образом, коэффициенты при $x^m a^l,$ $m=n,...,2n-1$, $l=n-k,...,2n-1-k$ в $g_1(x)$ и $(-1)^kg_2(x)$ равны.
Итого:
$$
g_1(x)-(-1)^kg_2(x) = \sum_{m=0}^{n-k-1} a^m x^{2n-1-k-m} (-1)^m C_{2n-1-k}^{m} - (-1)^k \sum_{m=0}^{n-1} x^m (-1)^{m} a^{2n-1-k-m} C_{2n-1-k}^m =
$$
$$
=\sum_{m=0}^{n-k-1} a^m x^{2n-1-k-m} (-1)^m C_{2n-1-k}^{m} + \sum_{m=n-k}^{2n-k-1} (-1)^{m} x^{2n-1-k-m} a^{m} C_{2n-1-k}^m =(x-a)^{2n-k-1}.
$$
\end{proof}

\begin{rim}
Теорема \ref{thm:splain} формально позволяет вычислять значения $A^2_{n,k}(a)$. Для этого, согласно формуле \eqref{eq:norm}, достаточно вычислить
$\|g_{n,k}\|^2_{\mathcal H}$. Однако вычисление производных сплайнов $g_{n,k}$, а потом интегрирование квадрата $n$-ой производной  представляет собой технически нелегкую задачу.
\end{rim}

\section{Вспомогательные формулы для вычисления $A^2_{n,k}$.}

Заметим, что $A^2_{n,k}(a)=\|g_{n,k}\|^2_{\mathcal H}=g^{(n)}_{n,k}(a)$, но вычисление величины $A^2_{n,k}(a)$ через производные  функций \eqref{eq:g}, как уже указывалось, приводит к громоздким выкладкам. Поэтому мы воспользуемся методом, предложенным в работе \cite{Kalyab}, который основан на свойствах многочленов Лежандра.

Рассмотрим  ортогональную  на отрезке $[0;1]$ систему смещенных многочленов Лежандра $\left\{ P_m\right\}_{m=0}^\infty$, определяемых формулой
$$
P_m(x)=\dfrac{1}{m!}\bigl((x^2-x)^m\bigr)^{(m)}, \quad m=0,1,\ldots
$$
Смещенные  полиномы Лежандра имеют следующую нормировку: $\|P_m\|^2_{L_2[0;1]}=\frac{1}{2m+1}$.

Первообразные порядка $l\geqslant 0$ таких полиномов понимаются следующим образом
$$
P_m^{(-l)}=\dfrac{1}{m!}\bigl((x^2-x)^m\bigr)^{(m-l)}.
$$

В \cite{Kalyab} была получена формула, связывающая функции  $A^2_{n,k}$,  $A^2_{n-k,0}$ с первообразными полиномов Лежандра. Для отрезка $[0;1]$ она выглядит следующим образом:
\begin{equation}\label{eq:An}
A^2_{n,k}(a)=A^2_{n-k,0}(a)-\sum_{m=n-k}^{n-1}\bigl( P_m^{(k-n)}(a)\bigr)^2(2m+1),
\end{equation}

Для $A^2_{n-1,k-1}$ соответственно получаем
\begin{equation}\label{eq:An-1}
A^2_{n-1,k-1}(a)=A^2_{n-k,0}(a)-\sum_{m=n-k}^{n-2}\bigl(P_m^{(k-n)}(a)\bigr)^2(2m+1).
\end{equation}

Рассмотрев разность \eqref{eq:An} и \eqref{eq:An-1}, получим
\begin{equation}\label{eq:razn}
A^2_{n,k}(a)-A^2_{n-1,k-1}(a)=-\bigl(P_{n-1}^{(k-n)}(a)\bigr)^2(2n-1).
\end{equation}

Таким образом, достаточно знать первообразную соответствующего порядка только одного полинома Лежандра. Первообразная полинома Лежандра порядка $n-k$ равна
$$
P_{n-1}^{(k-n)}(a)=\dfrac{1}{(n-1)!}\left(\bigl(a^2-a\bigr)^{n-1}\right)^{(k-1)}.
$$
Действуя аналогично, можно получить соотношение между $A^2_{n,k}(a)$ и $A^2_{n-2,k-2}(a)$, причем  индексы  у них имеют одинаковую четность, но  знать надо  первообразные уже двух полиномов Лежандра. Действительно:
$$
A^2_{n-1,k-1}(a)-A^2_{n-2,k-2}(a)=-\bigl( P_{n-2}^{(k-n)}(a)\bigr)^2(2n-3).
$$
Следовательно,
$$
A^2_{n,k}(a)=A^2_{n-2,k-2}(a)-\bigl(P_{n-2}^{(k-n)}(a)\bigr)^2(2n-3)-
\bigl( P_{n-1}^{(k-n)}(a)\bigr)^2(2n-1).
$$

Для упрощения дальнейших вычислений введем новую переменную $t:=a^2-a$. Заметим, что для произвольной дифференцируемой функции $f$ справедливо соотношение $\dfrac{df(t)}{da}=(2a-1)f'(t)$, которым в дальнейшем мы будем часто пользоваться.
Условие $a\in[0;1]$ влечет  условие $t\in\left[-\frac14;0\right]$.

\begin{lem}\label{lem:tpl}
Для каждого $l=0,1,\ldots$ такого, что $k\leqslant m$, справедливы соотношения
\begin{gather*}
\dfrac{d^k(t^m)}{da^k}=\sum_{j=0}^l a_{lj} t^{m-l-j}(2a-1)^{2j}, \quad k=2l;\\
\dfrac{d^k(t^m)}{da^k}=(2a-1)\sum_{j=0}^l b_{lj} t^{m-l-j-1}(2a-1)^{2j}, \quad k=2l+1,
\end{gather*}
где $a_{lj}$, $b_{lj}$ некоторые числовые коэффициенты.
\end{lem}
\begin{proof}
1) При $k=0$ и $k=1$ утверждение очевидно.

2) Далее  достаточно посмотреть, что происходит при однократном  дифференцировании со слагаемым вида
$t^{m-l-j}(2a-1)^{2j}$ (при четных $k$) или $t^{m-l-j-1}(2a-1)^{2j}$ (при нечетных $k$).
Пусть $k=2l$. Получаем:
\begin{multline*}
\dfrac{d(t^{m-l-j}(2a-1)^{2j}}{da}=(m-l-j)t^{m-l-j-1}(2a-1)^{2j+1}+4jt^{m-l-j}(2a-1)^{2j-1}=\\=
(2a-1)\left((m-l-j)t^{m-l-j-1}(2a-1)^{2j}+4jt^{m-l-j}(2a-1)^{2(j-1)}\right).
\end{multline*}
Видно, что слагаемое $t^{m-l-j}(2a-1)^{2j}$ преобразовалось в два слагаемых, отвечающих правилу дифференцирования при $k=2l+1$.

Случай нечетного $k$ проверяется аналогично.

Заметим, что максимальная степень переменной $t$ в суммах равна $m-l$ ($k=2l$) или $m-l-1$ ($k=2l+1$), а минимальная равна $m-k$.

Следовательно,  $\left(\dfrac{d^k(t^m)}{da^k}\right)^2$ есть многочлен степени $2m-k$, а минимальная степень переменной $t$ у него равна $2(m-k)$.
\end{proof}

\begin{lem}\label{lem:PL1}
Для каждого $k=1,2,\ldots, n-1$ многочлен $\left(P_{n-1}^{(k-n)}\right)^2$
есть многочлен степени $2n-k-1$ от $t$. Минимальная степень $t$ в этом многочлене равна $2(n-k)$.
\end{lem}

\begin{proof}
1) Для $k=1$ это утверждение очевидно:
$$
\left(P_{n-1}^{(1-n)}(x)\right)^2=\left(\dfrac{1}{(n-1)!}(x^2-x)^{n-1}\right)^2=\dfrac{t^{2n-2}}{((n-1)!)^2}.
$$

2) Далее заметим, что
$$
\left(P_{n-1}^{(k-n)}(x)\right)^2=\left(\dfrac{1}{(n-1)!}\bigl((x^2-x)^{n-1}\bigr)^{(k-1)}\right)^2=
\dfrac{1}{((n-1)!)^2}\left(\dfrac{d^{k-1}t^{n-1}}{dx^{k-1}}\right)^2.
$$

Применение леммы \eqref{lem:tpl} заканчивает доказательство.
\end{proof}

\begin{thm}
В переменной $t=a^2-a$ величины  $A^2_{n,k}$ можно представить в виде
$$
A^2_{n,k}(a)=-\dfrac{t^{2n-2k-1}}{((n-k)!)^2(2n-2k-1)}\cdot B_{n,k}(t),
$$
где $B_{n,k}$ --- многочлены степени $k$, свободный член которого равен $(n-k)^2$.
\end{thm}
\begin{proof}
1) Для $A^2_{n,0}$ такое представление очевидно выполнено, так как
$$
A^2_{n,0}(a)=\dfrac{(a(1-a))^{2n-1}}{((n-1)!)^2(2n-1)}=-\dfrac{t^{2n-1}}{((n)!)^2(2n-1)}\cdot n^2,
$$
т.е. $B_{n,0}=n^2$.

2) Из представления \eqref{eq:razn} и леммы \eqref{lem:PL1} следует, что
\begin{multline}
A^2_{n,k}=A^2_{n-1,k-1}-\bigl(P_{n-1}^{(k-n)}(a)\bigr)^2(2n-1)=\\=
\dfrac{-t^{2n-2k-1}}{((n-k)!)^2(2n-2k-1)}\cdot B_{n-1,k-1}(t)-
t^{2n-2k}\tilde B_{n,k}(t)=\dfrac{-t^{2n-2k-1}}{((n-k)!)^2(2n-2k-1)}\left(B_{n-1,k-1}(t)+t\tilde B_{n,k}(t)\right),
\end{multline}
где степень многочлена $\tilde B_{n,k}$ равна $k-1$.

Многочлен $B_{n,k}:=B_{n-1,k-1}(t)+t\tilde B_{n,k}(t)$ имеет степень $k$, а его свободный член совпадает со свободным членом многочлена $B_{n-1,k-1}$. Поскольку $(n-1)-(k-1)=n-k$, то применение метода математической индукции завершает доказательство.
\end{proof}


\section{Константы вложения $\Lambda^2_{n,3}$}\label{sub:An3}

Преобразуем полученную в \cite{Kalyab} для  $A^2_{n,2}(a)$ формулу  на отрезок $[0;1]$:

\begin{multline*}
A^2_{n,2}(a)=\dfrac{(1-a)^{2n-5}a^{2n-5}}{((n-2)!)^2(2n-5)}\left((n-2)^2-4(n-1)(2n-5)a +4(3n-2)(2n-5)a^2-\right.
\\-
\left.8(2n-1)(2n-5)a^3+4(2n-1)(2n-5)a^4\right).
\end{multline*}

В переменной $t$:
$$
A^2_{n,2}(t)=\dfrac{-t^{2n-5}}{((n-2)!)^2(2n-5)}((n-2)^2+4(n-1)(2n-5)t+4(2n-1)(2n-5)t^2).
$$

Вычислим $P_{n-1}^{(3-n)}(t)$:
$$
P_{n-1}^{(3-n)}(t)=\dfrac{1}{(n-1)!}\dfrac{d^2t^{n-1}}{da^2}=\dfrac{t^{n-3}}{(n-2)!}(2(2n-3)t+n-2).
$$

Из формулы \eqref{eq:razn} после упрощения получаем,

\begin{multline}\label{eq:An3}
A^2_{n,3}=\dfrac{-t^{2n-7}}{((n-2)!)^2(2n-7)}
\left(4(2n-1)(2n-3)^2(2n-7)t^3+12(n-1)(n-2)(2n-3)(2n-7)t^2+\right.\\+\left.
3(n-2)^2(2n-3)(2n-7)t+(n-2)^2(n-3)^2\right).
\end{multline}

Множитель $\dfrac{1}{((n-2)!)^2(2n-7)}$ не влияет на точки максимума и минимума функции $A^2_{n,3}$.

Для упрощения вычислений введем функцию

$$
f_{n,3}(t):=A^2_{n,3}\cdot ((n-2)!)^2(2n-7).
$$

Тогда
$$
\dfrac{df_{n,3}}{da}=-t^{2n-8}(2a-1)\bigl(tf_{n,3}'(t)+(2n-7)f_{n,3}(t)\bigr),
$$
где
$$
f_{n,3}'(t)=12(2n-3)^2(2n-1)(2n-7)t^2+24(n-1)(n-2)(2n-3)(2n-7)t+3(n-2)^2(2n-3)(2n-7).
$$

После приведения подобных слагаемых получаем:
\begin{multline*}
\dfrac{df_{n,3}}{da}=-t^{2n-8}(2a-1)(2n-7)(n-2)\left[8(2n-1)(2n-3)^2t^3+12(n-1)(2n-3)(2n-5)t^2+\right.\\+\left.
6(n-2)(n-3)(2n-3)t+(n-2)^2(n-3)^2\right].
\end{multline*}

Кубический многочлен, стоящий в квадратных скобках в предыдущей формуле,
раскладывается на множители следующим образом:
$$
\bigl(2(2n-3)t+(n-3)\bigr)\cdot\bigl(4(2n-3)(2n-1)t^2+4(n-2)(2n-3)t+(n-2)(n-3)\bigr),
$$
откуда получаем его корни
\begin{gather*}
t_1(n):=\dfrac{-(n-2)(2n-3)-\sqrt{3(n-2)(2n-3)}}{2(2n-1)(2n-3)},\quad
t_2(n):=\dfrac{-(n-2)(2n-3)+\sqrt{3(n-2)(2n-3)}}{2(2n-1)(2n-3)}\,\\
t_3(n):=-\dfrac{n-2}{2(2n-1)}.
\end{gather*}
Для простоты, там где это не важно, зависимость $t_i$ от $n$ мы указывать не будем.
Непосредственно проверяется, что для этих корней выполнены неравенства
$$
-\dfrac14<t_1<t_3<t_2<0.
$$

По точкам  $t_i$ несложно найти точки экстремума функции $f$ (напомним, что точка $a=\frac12$ также является нулем $f'_{n,3}$):
$$
a_i^{\pm}=\dfrac{1}{2}\pm\dfrac{1}{2}\sqrt{1+4t_i},
$$
из которых
$$
a=\frac12,\quad a_3^{\pm}=\dfrac{1}{2}\pm\dfrac{1}{2}\sqrt{1+4t_3}
$$
--- точки минимума, а
$$
a_1^{\pm}=\dfrac{1}{2}\pm\dfrac{1}{2}\sqrt{1+4t_1}\quad \text{и}\quad a_2^{\pm}=\dfrac{1}{2}\pm \dfrac{1}{2}\sqrt{1+4t_2}
$$
--- точки максимума.

Чтобы определить точную константу вложения, надо определить, какое из двух значений $f_{n,3}(t_1)$, $f_{n,3}(t_2)$ является максимальным.

Подставив $t_1$ и $t_2$ в $f_{n,3}$ и упростив полученные выражения, получим
\begin{gather*}
f_{n,3}(t_1)=-t_1^{2n-7}\dfrac{3(n-2)\left(6n^2-27n+30-\sqrt 3(2n-7)\sqrt{2n^2-7n+6}\right)}{2(2n-1)^2},\\
f_{n,3}(t_2)=-t_2^{2n-7}\dfrac{3(n-2)\left(6n^2-27n+30+\sqrt 3(2n-7)\sqrt{2n^2-7n+6}\right)}{2(2n-1)^2}.
\end{gather*}

Рассмотрим отношение
$$
\dfrac{f_{n,3}(t_1)}{f_{n,3}(t_2)}=
\left(\dfrac{t_1}{t_2}\right)^{2n-7}
\dfrac{6n^2-27n+30-\sqrt 3(2n-7)\sqrt{2n^2-7n+6}}
{6n^2-27n+30+\sqrt 3(2n-7)\sqrt{2n^2-7n+6}}.
$$

При $n\geqslant 4$ функция
$$
g(n):=\left(\dfrac{t_1(n)}{t_2(n)}\right)^{2n-7}=
\left(1+\dfrac{2\sqrt{3(n-2)(2n-3)}}{(n-2)(2n-3)-\sqrt{3(n-2)(2n-3)}}\right)^{2n-7}
$$
\textit{возрастая} стремится к $e^{2\sqrt 6}$ при $n\to+\infty$.

При тех же $n$ функция
$$
h(n):=\dfrac{6n^2-27n+30-\sqrt 3(2n-7)\sqrt{2n^2-7n+6}}
{6n^2-27n+30+\sqrt 3(2n-7)\sqrt{2n^2-7n+6}}
$$
\textit{убывая} стремится к $5-2\sqrt 6>0,1$ при $n\to+\infty$.

Так как $g(4)=\dfrac{13+2\sqrt{3}}{7}>2$,
$h(4)=\dfrac{59-6\sqrt{30}}{49}>\dfrac12$, $g(5)=\left(\dfrac{21+3\sqrt 7}{21-\sqrt 7}\right)^3=
\left(\dfrac{4+\sqrt 7}{3}\right)^3>10$, то
$$
\dfrac{f_{n,3}(t_1(n))}{f_{n,3}(t_2(n))}>1
$$
при всех $n\geqslant 4$. Следовательно, максимальное значение функция  $A^2_{n,3}$ принимает при
$a_1^{\pm}=\dfrac{1}{2}\pm\dfrac{1}{2}\sqrt{1+4t_1}$.

Само это значение равно
\begin{equation}\label{eq:const_emb}
\left(\dfrac{(n-2)(2n-3)+\sqrt{(n-2)(2n-3)}}{(2n-1)(2n-3)}\right)^{2n-7}\cdot
\dfrac{(n-2) \left(3(n-2)(2n-5)-\sqrt{3}(2n-7)\sqrt{(n-2)(2n-3)}\right)}
{4^{n-3}((n-2)!)^2(2n-1)^2(2n-7)}.
\end{equation}

\bigskip

\section{Константы вложения $\Lambda^2_{n,5}$}

Опираясь на явный вид $A^2_{n,3}$ \eqref{eq:An3} и используя рекуррентную формулу \eqref{eq:An-1}, найдем
вид  $A^2_{n-1,4}$ на отрезке $[0;1]$:
\begin{multline*}
A^2_{n-1,4}(t)=-\dfrac{t^{2n-11}}{((n-4)!)^2(2n-11)}\left(16(2n-3)(2n-5)^2(2n-11)t^4+\right.\\+
32(n-2)(2n-5)(2n-7)(2n-11)t^3+
24(n-3)(n-4)(2n-5)(2n-11)t^2+\\+\left.
8(n-3)(n-4)^2(2n-11)t+(n-4)^2(n-5)^2\right).
\end{multline*}

Первообразная порядка $(5-n)$ смещенного полинома Лежандра $P_{n-1}$ равна
$$
P_{n-1}^{(5-n)}=\dfrac{1}{(n-1)!)}\dfrac{d^4t^{n-1}}{d^4a}=
\dfrac{t^{n-5}}{(n-3)!}\left(4(2n-3)(2n-5)t^2+4(n-3)(2n-5)t+(n-3)(n-4)\right)
$$

Аналогично вычислениям, проведенным в разделе \ref{sub:An3}, применяя формулу \eqref{eq:razn}, найдем
\begin{multline}\label{mult:An5}
A^2_{n,5}=\dfrac{-t^{2n-11}}{((n-3)!)^2(2n-11)}\cdot\left[16(2n-1)(2n-3)^2(2n-5)^2(2n-11)t^5+\right.\\+
80(n-1)(n-3)(2n-3)(2n-5)^2(2n-11)t^4+40(n-2)(n-3)(2n-3)(2n-5)(2n-7)(2n-11)t^3+\\+
40(n-2)(n-3)^2(n-4)(2n-5)(2n-11)t^2+\\+\left.
5(n-3)^2(n-4)^2(2n-5)(2n-11)t+(n-3)^2(n-4)^2(n-5)^2\right].
\end{multline}
Введем функцию
$$
f_{n,5}(t):=A^2_{n,5}(t)\cdot ((n-3)!)^2(2n-11)
$$
и найдем ее производную
\begin{multline}\label{eq:deriv5}
\dfrac{d f_{n,5}}{da}=-t^{2n-12}(n-3)(2n-11)(2a-1)\cdot\left[32(2n-1)(2n-3)^2(2n-5)^2t^5+\right.\\+
80(n-1)(2n-3)(2n-5)^2(2n-7)t^4+80(n-2)(n-4)(2n-3)(2n-5)(2n-7)t^3+\\+
40(n-2)(n-3)(n-4)(2n-5)(2n-9)t^2+10(n-3)(n-4)^2(n-5)(2n-5)t+\\+\left.
(n-3)(n-4)^2(n-5)^2\right].
\end{multline}

Многочлен, стоящий в квадратных скобках формулы \eqref{eq:deriv5} раскладывается на множители:
\begin{multline}
\left[4(2n-3)(2n-5)t^2+4(n-4)(2n-5)t+(n-4)(n-5)\right]\cdot\\
\left[8(2n-1)(2n-3)(2n-5)t^3+12(n-3)(2n-3)(2n-5)t^2+6(n-3)(n-4)(2n-5)t+\right.\\+\left.
(n-3)(n-4)(n-5)\right].
\end{multline}
Обозначим их соответственно
\begin{gather*}
g_2(t):=4(2n-3)(2n-5)t^2+4(n-4)(2n-5)t+(n-4)(n-5),\\
g_3(t):=8(2n-1)(2n-3)(2n-5)t^3+12(n-3)(2n-3)(2n-5)t^2+6(n-3)(n-4)(2n-5)t+\\
+(n-3)(n-4)(n-5).
\end{gather*}

Исследуем нули $\dfrac{d f_{n,5}}{da}$.
$$
g_2(t)=0 \Leftrightarrow \hat t_{1,2}=-\dfrac{n-4}{2(2n-3)}\pm
\dfrac{\sqrt{5(n-4)(2n-5)}}{2(2n-3)(2n-5)}.
$$

Нумерацию корней будем вести в порядке возрастания. Несложно убедиться, что при $n\geqslant 6$ выполнены неравенства
$$
-\dfrac{1}{4}<\hat t_1<\hat t_2<0.
$$
Многочлен $g_3$ имеет три вещественных корня. Действительно,

1) $g_3\left(-\dfrac14\right)=-\dfrac{15}{8}<0$.

2) Прямой подстановкой проверяется справедливость  неравенств $g_3(\hat t_1)>0$, $g_3(\hat t_2)<0$. В  самом деле
$$
g_3(\hat t_1)=-\dfrac{10(n-4)\left[(2n-5)(2n-13)-(2n-7)\sqrt{5(n-4)(2n-5)}\right]}{(2n-3)^2(2n-5)^2}>0.
$$

Последнее неравенство выполнено при $n\geqslant 6$, поскольку знак $g_3(\hat t_1)$ противоположен знаку
разности $(2n-5)(2n-13)-(2n-7)\sqrt{5(n-4)(2n-5)}$, которая в свою очередь отрицательна:
\begin{multline*}
(2n-13)(2n-5)<(2n-7)\sqrt{5((n-4)(2n-5)} \Leftrightarrow(2n-13)^2(2n-5)-5(2n-7)^2(n-4)<0 \Leftrightarrow\\
\Leftrightarrow -3(n-5)(2n-3)^2<0.
\end{multline*}

С другой стороны
$$
g_3(\hat t_2)=-\dfrac{10(n-4)\left[(2n-5)(2n-13)+(2n-7)\sqrt{5(n-4)(2n-5)}\right]}{(2n-3)^2(2n-5)^2}<0.
$$

3) Наконец, $g_3(0)=(n-3)(n-4)(n-5)>0$ при $n\geqslant 6$.

Таким образом многочлен $g_3$ имеет три смены знака на отрезке $\left[-\dfrac14;0\right]$, а значит имеет ровно три корня на промежутке $\left(-\dfrac14;0\right)$. Обозначим их $t_i$, $i=1,2,3$ и пронумеруем в порядке возрастания: $-1/4<t_1<t_2<t_3<0$. Уточним расположение этих корней.

Из рассуждений, приведенных выше,  следует, что корни многочленов $g_2$ и $g_3$ чередуются:
$$
-\dfrac14<t_1<\hat t_1<t_2<\hat t_2<t_3<0.
$$

Переходя как в разделе \ref{sub:An3} к переменной $a$, находим, что
$$
a_0:=\dfrac12;\quad a^{(2),\pm}_i:=\dfrac12\pm\sqrt{1+4t^{(2)}_i}, \quad i=1,2  \text{ --- точки минимума},
$$

$$
a^{(3),\pm}_i:=\dfrac12\pm\sqrt{1+4t^{(3)}_i}, \quad i=1,2,3  \text{ --- точки максимума}.
$$

\subsection{Вычисление корней многочлена $g_3$.}
Найдем корни многочлена $g_3$. Обозначим его коэффициенты
\begin{gather*}
\mathbf{a}:=8(2n-1)(2n-3)(2n-5),\quad
\mathbf{b}:=12(n-3)(2n-3)(2n-5),\\
\mathbf{c}:=6(n-3)(n-4)(2n-5), \quad \mathbf{d}:=(n-3)(n-4)(n-5).
\end{gather*}
Классическая замена $t=y-\frac{\mathbf{a}}{\mathbf{b}}$ приводит многочлен $g_3$ к приведенному виду
$$
\tilde g_3(y):=y^3+py+q,
$$
где $p=\tfrac{\textbf{c}}{\mathbf{a}}-\tfrac{\mathbf{b^2}}{\mathbf{3a^2}}=-\frac{15(n-3)}{4(2n-1)^2(2n-3)}$,
$q=\frac{2}{27}\left(\frac{\mathbf{b}}{\mathbf{a}}\right)^3-\frac{\mathbf{bc}}{\mathbf{3a^2}}+
\frac{\mathbf{d}}{\mathbf{a}}=-\frac{5(n-3)(2n-11)}{4(2n-1)^3(2n-3)(2n-5)}$.

Определим величину
$$
\mathbf{Q}:=\left(\frac{p}{3}\right)^3+\left(\frac{q}{2}\right)^2.
$$
Подставив найденные значения $p$ и $q$, получим
$$
\mathbf{Q}=-\frac{125(n-3)^2(n-4)}{64(2n-1)^4(2n-5)^2}.
$$

Далее определим два числа
$$
\alpha:=\sqrt[3]{-\frac{q}{2}+\sqrt{\mathbf Q}},\qquad \beta:=\sqrt[3]{-\frac{q}{2}-\sqrt{\mathbf Q}}.
$$

Корни многочлена $\tilde g_3$ находятся по формулам (ветви корней для $\alpha$ и $\beta$ надо подбирать, так чтобы $\alpha\cdot\beta=-\frac{p}{3}$).
$$
y_1=\alpha+\beta,\quad y_{2,3}=-\dfrac{\alpha+\beta}{2}\pm i\dfrac{\alpha-\beta}{2}\sqrt 3.
$$

Числа $-\frac{q}{2}\pm\sqrt{\mathbf Q}$ можно представить в виде
$$
\dfrac{5(n-3)\sqrt{5(n-3)}}{8(2n-1)^3(2n-3)\sqrt{2n-3}}\left(\cos \varphi\pm i\sin\varphi\right),
$$
где
$$
\cos\varphi=\dfrac{2n-11}{2n-5}\dfrac{\sqrt{2n-3}}{\sqrt{5(n-3)}},\qquad
\sin\varphi=\dfrac{2n-1}{2n-5}\dfrac{\sqrt{3(n-4)}}{\sqrt{5(n-3)}}.
$$

Оценим $\tg \varphi=\dfrac{2n-1}{2n-11}\dfrac{\sqrt{3(n-4)}}{\sqrt{2n-3}}$.
Рассмотрим функцию $w(s):=\dfrac{2s-1}{2s-11}\dfrac{\sqrt{s-4}}{\sqrt{2s-3}}$ при $s\geqslant 6$. Ее производная
$$
w'(s)=-\dfrac{5(12s^2-64s+85)}{2(2s-11)^2\sqrt{(s-4)(2s-3)^3}}<0 \quad (\text{при } s\geqslant 6),
$$

поэтому $1<\sqrt{\dfrac32}=\sqrt3\lim_{n\to\infty} w(n)<\tg\varphi\leqslant \sqrt{3}w(6)=22$. Следовательно,
$\frac{\pi}{4}<\varphi<\frac{\pi}{2}$.

Таким образом, корни приведенного кубического трехчлена $\tilde g_3$ имеют вид
$$
y_1=\dfrac{\sqrt{5(n-3)}}{(2n-1)\sqrt{2n-3}}\cos\frac{\varphi}{3},\quad y_{2,3}=-\dfrac{\sqrt{5(n-3)}}{(2n-1)\sqrt{2n-3}}\cos\left(\frac{\varphi}{3}\pm\frac{\pi}{3}\right),
$$
при этом $\frac{\pi}{12}<\frac{\varphi}{3}<\frac{\pi}{6}$, $-\frac{\pi}{4}<\frac{\varphi}{3}-\frac{\pi}{3}<-\frac{\pi}{6}$,
$\frac{5\pi}{12}<\frac{\varphi}{3}+\frac{\pi}{3}<\frac{\pi}{2}$.
Так как
$$
-\cos\left(\frac{\varphi}{3}-\frac{\pi}{3}\right)<-\cos\left(\frac{\varphi}{3}+\frac{\pi}{3}\right)<0<\cos\frac{\varphi}{3},
$$
то в переменной $t$ корни многочлена $g_3$, с учетом упорядочения  $t_1<t_2<t_3$,  имеют вид
\begin{gather}\label{eq:t1}
t_1=-\dfrac{n-3}{2(2n-1)}-\dfrac{\sqrt{5(n-3)}}{(2n-1)\sqrt{2n-3}}\cos\left(\frac{\varphi}{3}-\frac{\pi}{3}\right),\\
\label{eq:t2}
t_2=-\dfrac{n-3}{2(2n-1)}-\dfrac{\sqrt{5(n-3)}}{(2n-1)\sqrt{2n-3}}\cos\left(\frac{\varphi}{3}+\frac{\pi}{3}\right),\\
\label{eq:t3}
t_3=-\dfrac{n-3}{2(2n-1)}+\dfrac{\sqrt{5(n-3)}}{(2n-1)\sqrt{2n-3}}\cos\frac{\varphi}{3}.
\end{gather}

\subsection{Определение глобального максимума функции $f_{n,5}$.}

Заметим, что вид точек максимума функций $A^2_{n,5}$ значительно более сложный по сравнению с точками максимума функций $A^2_{n,1}$ и $A^2_{n,3}$. Методы определения глобального максимума, примененные для  величин $A^2_{n,1}$, $A^2_{n,2}$ в работе \cite{Kalyab} и развитые для $A^2_{n,3}$ в разделе \ref{sub:An3} настоящей работы, оказываются малопригодными для величины $A^2_{n,5}$. Причины этого содержатся не только в громоздкости формул \eqref{eq:t1}--\eqref{eq:t3}. Кроме этого играет свою роль неявный вид угла $\frac{\varphi}{3}$ в этих формулах, а также тот факт, что с ростом параметра $k$ в константах вложения растет и степень многочлена, значения которого в точках максимума надо вычислять.

Наша гипотеза заключается в том, что именно точка $t_1$ является точкой глобального максимума многочлена  $f_{n,5}$, а значит и функции $A^2_{n,5}$. Чтобы это доказать, опираясь на методы работы \cite{Kalyab} и раздела \ref{sub:An3} данной работы, достаточно показать, что при всех $n\geqslant 6$ выполнены неравенства
$$
\dfrac{A^2_{n,5}(t_1)}{A^2_{n,5}(t_2)}>1,\qquad \dfrac{A^2_{n,5}(t_1)}{A^2_{n,5}(t_3)}>1.
$$

Для любой функции вида $f:=t^{m}\cdot p(t)$, где $p$ --- некоторая дифференцируемая функция, $m\in \mathbb{N}$ ее производная равна
$$
f'=mt^{m-1}p+t^mp'=t^{m-1}(mp+tp').
$$
На любом промежутке, не содержащем ноль, $f'=0$ $\Leftrightarrow$ $mp+tp'=0$. Стало быть,
если $t^*$ произвольный  нуль производной $f$, то для него выполнено соотношение
$$
p(t^*)=-\dfrac{t^*p'(t^*)}{m}.
$$

Так как функция  $A^2_{n,5}$ имеет вид:
$$
A^2_{n,5}=-t^{2n-11}\cdot p_5\cdot Const,
$$
где $p_5$ --- многочлен пятой степени, определенный в \eqref{mult:An5},
то для   функции $A^2_{n,5}$ в нулях ее производной выполнено соотношение
\begin{equation}\label{eq:relA}
\dfrac{A^2_{n,5}(t_i)}{A^2_{n,5}(t_j)}=\dfrac{(-t_i)^{2n-11}}{(-t_j)^{2n-11}}\cdot\dfrac{p_5(t_i)}{p_5(t_j)}=
\dfrac{(t_i)^{2n-10}}{(t_j)^{2n-10}}\cdot\dfrac{p'_5(t_i)}{p'_5(t_j)}.
\end{equation}
Заметим, что поскольку $A^2_{n,5}>0$ ($\Leftrightarrow$ $p_5>0$ на промежутке $(-\frac14;0)$), то
$p'_5(t^*)>0$ (т.е. во всех точках $t_i$).

\subsubsection{Экспоненциальная часть отношения \eqref{eq:relA} $t_i=t_1$, $t_j=2$.}\label{sub:exp12}

Выражение
$$
\dfrac{t_1}{t_2}=\dfrac{t_2+(t_1-t_2)}{t_2}=1+\dfrac{t_1-t_2}{t_2}
$$
можно привести к виду
$$
\dfrac{a+b\cos({\varphi}/{3}-{\pi}/{3})}{a+b\cos({\varphi}/{3}+{\pi}/{3})}=
1+\dfrac{\sqrt 3 b\sin\frac{\varphi}{3}}{a+b\cos(\varphi/3+\pi/3)},
$$
где
$$
a=a_n:=\dfrac{n-3}{2(2n-1)}\nearrow \frac14,\qquad
b=b_n:=\dfrac{\sqrt{5(n-3)}}{(2n-1)\sqrt{2n-3}}\searrow 0 \text{ при } n\to\infty.
$$
Несложно проверить, что
$$
b_n\sim\frac{\sqrt5}{2\sqrt 2}\frac1n,\quad \sin\frac{\varphi}{3}\searrow\sin\left(\frac13\arcsin\sqrt{\frac35}\right) \text{ при }n\to\infty.
$$

Тогда
$$
\left(\dfrac{t_1}{t_2}\right)^{2n-10}\nearrow e^{2\sqrt{30}\sin(1/3\arcsin\sqrt{3/5})}\geqslant 24.
$$

\subsubsection{Отношение $\dfrac{p'_5(t_i)}{p'_5(t_j)}$.}

Определим
\begin{multline*}
q(t):=\dfrac{p'_5(t)}{5(2n-5)(2n-11)}=
16(2n-1)(2n-3)^2(2n-5)t^4+64(n-1)(n-3)(2n-3)(2n-5)t^3+\\+
24(n-2)(n-3)(2n-3)(2n-7)t^2+
16(n-2)(n-3)^2(n-4)t+(n-3)^2(n-4)^2.
\end{multline*}
Многочлен $q$ раскладывается в произведение двух квадратных трехчленов:
$q=q_1\cdot q_2$, где
\begin{gather*}
q_1(t)=4(2n-3)(2n-5)t^2+4(n-3)(2n-5)t+(n-3)(n-4),\\
q_2(t)=4(2n-1)(2n-3)t^2+4(n-3)(2n-3)t+(n-3)(n-4).
\end{gather*}
Обозначим корни $q_1$ через $w_{1,2}$, а нули $q_2$ через $z_{1,2}$:
\begin{gather*}
w_1:=-\dfrac{n-3}{2(2n-3)}-\dfrac{\sqrt{3(n-3)}}{2(2n-3)\sqrt{2n-5}}\qquad
w_2:=-\dfrac{n-3}{2(2n-3)}+\dfrac{\sqrt{3(n-3)}}{2(2n-3)\sqrt{2n-5}},\\
z_1:=-\dfrac{n-3}{2(2n-1)}-\dfrac{\sqrt{5(n-3)}}{2(2n-1)\sqrt{2n-3}}\qquad
z_2:=-\dfrac{n-3}{2(2n-1)}+\dfrac{\sqrt{5(n-3)}}{2(2n-1)\sqrt{2n-3}}.
\end{gather*}
Покажем, что всегда выполнены неравенства
$$
w_1<z_1<w_2<z_2.
$$

В этом можно убедиться, вычислив
\begin{gather*}
q_2(w_2)=-\dfrac{4(n-3)(-3+\sqrt{3(n-3)(2n-5)})}{(2n-3)(2n-5)}<0,\\
q_1(z_1)=\dfrac{4(n-3)\left[(2n-3)(2n-11)-(2n-5)\sqrt{5(n-3)(2n-3)}\right]}{(2n-1)^2(2n-3)}<0.
\end{gather*}
Последнее неравенство выполнено поскольку
$$
(2n-3)(2n-11)<(2n-5)\sqrt{5(n-3)(2n-3)} \Leftrightarrow -3(n-4)(2n-1)^2<0 \text{ при } n\geqslant 6.
$$

Заметим, что поскольку $A^2_{n,5}>0$ ($\Leftrightarrow$ $p_5>0$ на промежутке $(-\frac14;0)$), то
$p'_5>0$ во всех точках $t_i$).

Тогда из соображения неравенств для $t_i$ и расстановки знаков, следует (см. рис.), что
$$
t_1<w_1<z_1<t_2<w_2<z_2<t_3.
$$

\begin{picture}(230,130)
\put(0,40){\vector(1,0){230}}\put(228,30){$t$}
\qbezier(20,50)(115,0)(200,80)
\put(20,55){$q_1$}
\qbezier(20,75)(120,-20)(200,80)
\put(20,80){$q_2$}
\put(30,35){$_{w_1}$}
\put(40,40){\circle*{3}}
\put(70,45){$_{z_1}$}
\put(68,40){\circle*{3}}
\put(128,46){$_{w_2}$}
\put(139,40){\circle*{3}}
\put(152,35){$_{z_2}$}
\put(155,40){\circle*{3}}
\put(195,25){$0$}
\multiput(200,35)(0,7){8}{\line(0,1){3}}
\put(12,35){$_{t_1}$}
\put(12,40){\circle*{3}}
\put(90,45){$_{t_2}$}
\put(90,40){\circle*{3}}
\put(170,45){$_{t_3}$}
\put(170,40){\circle*{3}}
\end{picture}

Обозначим вершины квадратных трехчленов $q_1$ и $q_2$ соответственно
$$
w_0:=-\dfrac{n-3}{2(2n-3)},\qquad z_0:=-\dfrac{n-3}{2(2n-1)}.
$$
Непосредственной подстановкой проверяется, что $q_1(w_0)=q_2(w_0)$.

Заметим, что
$$
|q_1(t_2)|\leqslant |q_1(w_0)|=\dfrac{3(n-3)}{2n-3}\nearrow \dfrac32,\qquad
|q_2(t_2)|\leqslant |q_2(z_0)|=\dfrac{5(n-3)}{2n-1} \nearrow \dfrac52.
$$

\subsubsection{Отношение $\frac{q_2(t_1)}{q_2(t_2)}$.}\label{sub:t1t2}

Перепишем многочлен $q_2$ в виде
$$
q_2(t)=4(2n-1)(2n-3)(t-z_0)^2-\dfrac{5(n-3)}{2n-1},\quad z_0=-\dfrac{(n-3)}{2(2n-1)}.
$$
Заметим, что
$$
t_i=z_0-b\cos\varphi_i,\quad b=\dfrac{\sqrt{5(n-3)}}{(2n-1)\sqrt{2n-3}},\quad
\varphi_1=\frac{\varphi}{3}-\frac{\pi}{3}, \quad \varphi_2=\frac{\varphi}{3}+\frac{\pi}{3},
\quad \varphi_3=\frac{\varphi}{3}-\pi.
$$
поэтому
$$
q_2(t_i)=4(2n-1)(2n-3)b^2\cos^2\varphi_i+10z_0,\quad b^2=\dfrac{5(n-3)}{(2n-1)^2(2n-3)}.
$$
Подставим, получим
$$
q_2(t_i)=\dfrac{20(n-3)}{2n-1}\cos^2\varphi_i-\dfrac{5(n-3)}{2n-1}=\dfrac{5(n-3)}{2n-1}(4\cos^2\varphi_i-1).
$$

Тогда
$$
\dfrac{q_2(t_1)}{\left|q_2(t_2)\right|}=\dfrac{4\cos^2\varphi_1-1}{1-4\cos^2\varphi_2}.
$$

Исследуем производную:
$$
\dfrac{d}{dn}\dfrac{4\cos^2\varphi_1-1}{1-4\cos^2\varphi_2}=
\dfrac{-8\cos\varphi_1\sin\varphi_1\frac13\varphi'(1-4\cos^2\varphi_2)-
(8\cos\varphi_2\sin\varphi_2\frac13\varphi')(4\cos^2\varphi_1-1)}{(1-4\cos^2\varphi_2)^2}.
$$
Исследуем знак числителя. Одна группа слагаемых имеет вид
$$
8\cos\varphi_2\sin\varphi_2\frac13\varphi'-8\cos\varphi_1\sin\varphi_1\frac13\varphi'=
\dfrac{8\varphi'}{3}\sin(\varphi_2-\varphi_1)\cos(\varphi_2+\varphi_1)=\dfrac{4\sqrt 3\varphi'}{3}\cos\frac{2\varphi}{3}.
$$
Другая группа слагаемых:
\begin{multline*}
\dfrac{32\varphi'\cos\varphi_1\cos\varphi_2}{3}(\sin\varphi_1\cos\varphi_2-\sin\varphi_2\cos\varphi_1)=
\dfrac{16\varphi'\left[\cos(\varphi_1+\varphi_2)+\cos(\varphi_1-\varphi_2)\right]}{3}\sin(\varphi_1-\varphi_2)=\\
=-\dfrac{8\sqrt3\varphi'\left[\cos\frac{2\varphi}{3}-\frac12\right]}{3}.
\end{multline*}

Собрав все вместе, получаем
$$
\dfrac{4\sqrt3\varphi'}{3}(1-\cos\frac{2\varphi}{3})<0,
$$
поскольку
$$
\varphi'=-\dfrac{\sqrt 3(6n-17)}{2(n-3)(2n-5)\sqrt{(n-4)(2n-3)}}<0.
$$

Следовательно,
$$
\dfrac{q_2(t_1)}{\left|q_2(t_2)\right|}\searrow \text{ при } n\to\infty.
$$

Точнее
$$
\dfrac{q_2(t_1)}{\left|q_2(t_2)\right|}\searrow
\dfrac{4\cos^2(\frac{\arccos\sqrt{\frac25}-\pi}{3})-1} {1-4\cos^2(\frac{\arccos\sqrt{\frac25}+\pi}{3})}
\geqslant 1.42.
$$

\subsubsection{Отношение $\frac{q_1(t_1)}{q_1(t_2)}$.}\label{sub:q12}

Мы уже указывали, что $|q_1(t_2)|\leqslant |q_1(w_0)|=\dfrac{3(n-3)}{2n-3}\nearrow \dfrac32$. Определим поведение $q_1(t_1)$. Величину $q_1(t_1)$ можно представить в виде
\begin{multline*}
q_1(t_1)=\dfrac{20(n-3)(2n-5)}{2n-1)^2}\cos^2\varphi_1-
\dfrac{8(n-3)(2n-5)\sqrt{5(n-3)}}{(2n-1)^2\sqrt{2n-3}}\cos\varphi_1-
\dfrac{(n-3)(2n+13)}{(2n-1)^2}=\\=
\dfrac{(n-3)(2n-5)}{2n-1)^2}\left[20\cos^2\varphi_1-
\dfrac{8\sqrt{5(n-3)}}{\sqrt{2n-3}}\cos\varphi_1-\dfrac{2n+19}{2n-5}\right].
\end{multline*}

Множитель $\dfrac{(n-3)(2n-5)}{(2n-1)^2}$ есть возрастающая функция при $n\to+\infty$. Рассмотрим функцию, стоящую в квадратных скобках
$$
r(n):=20\cos^2\varphi_1-
\dfrac{8\sqrt{5(n-3)}}{\sqrt{2n-3}}\cos\varphi_1-\dfrac{2n+19}{2n-5}.
$$
Ее производная (по переменной $n$) имеет вид
$$
r'(n)=-40\sin\varphi_1\cos\varphi_1\frac{\varphi'}{3}+
\frac{8\sqrt{5(n-3)}}{\sqrt{2n-3}}\sin\varphi_1\frac{\varphi'}{3}+
\frac{12\sqrt{5}}{\sqrt{n-3}(2n-3)^{3/2}}\cos\varphi_1+\frac{48}{(2n-5)^2}.
$$

Сгруппировав слагаемые
$$
r'(n)=\left(\frac{48}{(2n-5)^2}-10\sin(2\varphi_1)\frac{\varphi'}{3}\right)+
\cos\varphi_1\left(\frac{12\sqrt{5}}{\sqrt{n-3}(2n-3)^{3/2}}-20\sin\varphi_1\frac{\varphi'}{3}\right)
+\dfrac{8\sqrt{5(n-3)}}{\sqrt{2n-3}}\sin\varphi_1\frac{\varphi'}{3}
$$
и учитывая оценки $|\sin\varphi_1|<\frac{\sqrt{2}}{2}$, а также явный вид $\varphi'$, получаем, что $r'(n)>0$, т.е. $q_1(t_1)$ возрастает как произведение двух положительных возрастающих функций.

Численные методы показывают, что $q_1(t_1)$ растет очень медленно при $n\to+\infty$ и принимает значения, близкие к нулю. Вычисления показывают, что при всех $n\geqslant 6$ выполнено $q_1(t_1)>0,2$.

Из полученных оценок и монотонности функций следует, при всех $n\geqslant 6$ выполнено
$$
\dfrac{q_1(t_1)q_2(t_1)}{q_1(t_2)q_2(t_2)}>\frac23\cdot 1,42\cdot 0,2>0,188.
$$

Экспоненциальная множитель $\left(\frac{t_1}{t_2}\right)^{2n-10}$ при $n\geqslant 8$ удовлетворяет неравенству
$\left(\frac{t_1}{t_2}\right)^{2n-10}>6$.

Таким образом, достаточно получить отношение $\dfrac{q_1(t_1)q_2(t_1)}{q_1(t_2)q_2(t_2)}$ при $n=6$ и $n=7$.

Имеем
$$
\left.\dfrac{q_1(t_1)q_2(t_1)}{q_1(t_2)q_2(t_2)}\right|_{n=6}>0,46,\quad
\left.\left(\frac{t_1}{t_2}\right)^{2n-10}\right|_{n=6}>2,8;\qquad
\left.\dfrac{q_1(t_1)q_2(t_1)}{q_1(t_2)q_2(t_2)}\right|_{n=7}>0,38,\quad
\left.\left(\frac{t_1}{t_2}\right)^{2n-10}\right|_{n=7}>4,6.
$$

Окончательно получаем, что  при всех $n\geqslant 6$ справедливо неравенство $\dfrac{A^2_{n,5}(t_1)}{A^2_{n,5}(t_2)}>1$.

Аналогично показывается, что  при $n\geqslant 6$   справедливо неравенство  $\dfrac{A^2_{n,5}(t_1)}{A^2_{n,5}(t_3)}>1$. Следующий пункт посвящен установлению этого неравенства.

\subsubsection{Экспоненциальная часть отношения \eqref{eq:relA} $t_i=t_1$, $t_j=3$.}

Аналогично вычислениям, проведенным в \ref{sub:exp12}, получаем
$$
\frac{t_1}{t_3}=1+\frac{b\sqrt{3}\cos\left(\frac{\varphi}{3}-\frac{\pi}{6}\right)}{a-b\cos(\varphi/3)}.
$$
Тогда
$$
\left(\frac{t_1}{t_3}\right)^{2n-10}\nearrow e^{2\sqrt{30}\cos\left(1/3\arcsin\sqrt{\frac35}-\pi/6\right)}\geqslant 43\,060.
$$

\subsubsection{Отношение $\frac{q_2(t_1)}{q_2(t_3)}$.}\label{sub:t1t3}

Аналогично вычислениям раздела \ref{sub:t1t2} получаем
$$
\frac{q_2(t_1)}{q_2(t_3)}=\frac{4\cos^2\varphi_1-1}{4\cos^2\varphi_3-1}.
$$

Производная этого отношения равна
$$
\frac{d}{dn}\frac{q_2(t_1)}{q_2(t_3)}=
\frac{4\sqrt{3}\varphi'\cos\left(\frac{2\varphi}{3}-\frac{\pi}{3}\right)}{3}<0.
$$

Следовательно
$$
\frac{q_2(t_1)}{q_2(t_3)}=\frac{4\cos^2\varphi_1-1}{4\cos^2\varphi_3-1}\searrow
\dfrac{4\cos^2(\frac{\arccos\sqrt{\frac25}-\pi}{3})-1} {4\cos^2(\frac{\arccos\sqrt{\frac25}}{3})-1}\geqslant 0,42.
$$

\subsubsection{Отношение $\frac{q_1(t_1)}{q_1(t_3)}$.}\label{sub:q1t1t3}

Поведение $q_1(t_1)$ при $n\to\infty$ мы уже оценивали в п. \ref{sub:q12}. Рассмотрим теперь $q_1(t_3)$. Изучение этого значения оказалось непростой задачей. Мы исследуем более простую величину. В п. \ref{sub:exp12}.
были получены выражения для $t_i$ через величины $a$, $b$ и $\cos\varphi_i$, $i=1,2,3$.
Так как $0<\cos\varphi_3<1$, то
$$
q_1(t_3)<q_1(-a+b)=\dfrac{n-3}{(2n-1)^2}\left(\dfrac{8(2n-5)\sqrt{5(n-3)}}{\sqrt{2n-3}}+38n-119\right).
$$
Определение величин $a$ и $b$ дано в п. \ref{sub:exp12}. Вычислим производную $q_1(-a+b)$ по $n$:
$$
\frac{dq_1(-a+b)}{dn}=\frac{5(78n-239)}{(2n-1)^3}+
\frac{4 \sqrt{5(n-3)}(84n^2-340n+309)}{(2n-1)^3(2n-3)^{3/2}}\geqslant 0\text{ при }n\geqslant 6.
$$

Так как $\lim\limits_{n\to\infty}q_1(-a+b)=4\sqrt{2,5}+19/2<16$, то при всех
$n\geqslant 6$ выполнено неравенство $q_1(t_3)<16$.

Таким образом, при всех $n\geqslant 6$ справедлива оценка
$$
\frac{q_1(t_1)q_2(t_1)}{q_1(t_3)q_2(t_3)}>\frac{0,2\cdot 0,42}{16}=0,00525.
$$
Экспоненциальный множитель оценивается следующим образом
$$
\left.\left(\frac{t_1}{t_3}\right)^{2n-10}\right|_{n=6}>52;\quad \left.\left(\frac{t_1}{t_3}\right)^{2n-10}\right|_{n=7}>251
$$
и этих оценок хватает, чтобы прийти к выводу, что при $n\geqslant 7$ верно
$$
\frac{A^2_{n,5}(t_1)}{A^2_{n,5}(t_3)}>1.
$$
Для $n=6$ используемая оценка $q_1(t_3)$ недостаточна. Более точные вычисления показывают, что верно неравенство
$$
\left.\frac{q_1(t_1)q_2(t_1)}{q_1(t_3)q_2(t_3)}\right|_{n=6}>0,5,
$$
откуда следует, что
при $n\geqslant 6$ верно
$$
\frac{A^2_{n,5}(t_1)}{A^2_{n,5}(t_3)}>1.
$$
Таким образом справедлива теорема

\begin{thm}
Точка $t_1$ является точкой глобального максимума многочлена  $f_{n,5}$, а значит и функции $A^2_{n,5}$.
\end{thm}

\begin{rim}
В работе \cite{NazM} при вычислении $A^2_{n,2}$  и $A^2_{n,4}$ на отрезке $[-1;1]$ авторы также столкнулись со сравнением $A^2_{n,k}(0)$ со значением в другой точке максимума $t^*$. При этом  ошибочно предполагалось, что на некоторых промежутках выполнено неравенство
$$
\dfrac{p'_4(0)}{|p'_4(t^*)|}>1,
$$
что неверно.
Многочлен $p'_4$, определенный в работе \cite{NazM}, является аналогом многочлена $p'_5$ (см. \eqref{eq:relA}), но соответствует функции $A^2_{n,4}(t)$.
\end{rim}

\subsection{Точное значение константы вложения $\Lambda^2_{n,5}$}

Введем обозначение $B:=2\sqrt{5(n-3)}\cos \varphi_1+(n-3)\sqrt{2n-3}$. Тогда $f_{n,5}(t_1)$ равно
\begin{multline*}
\dfrac{1}{2^{2n-11}} \left(\dfrac{B}{(2n-1)\sqrt{2n-3}}\right)^{2n-11}\left[(n-3)^2 (n-4)^2(n-5)^2+\right.\\+
\dfrac{5(n-3)^2(n-4)^2(2n-5)(2n-11)\cdot B}{2(2n-1)\sqrt{2n-3}}-
\dfrac{10(n-2)(n-3)^2(n-4)(2n-5)(2n-11)\cdot B^2}{(2n-1)^2(2n-3)}+\\+
\dfrac{10(n-2)^2(n-3)(2n-5)(2n-7)(2n-11)\cdot B^3}{(2n-1)^3(2n-3)\sqrt{2n-3}}-
\dfrac{5(n-1)(n-3)(2n-5)^2(2n-11)\cdot B^4}{(2n-1)^4(2n-3)}+\\+\left.
\dfrac{(2n-5)^2\cdot B^5}{2(2n-1)^4\sqrt{2n-3}}\right].
\end{multline*}

Точная константа вложения $\Lambda^2_{n,5}$ соответственно равна
$$
\Lambda^2_{n,5}=\frac{f_{n,5}(t_1)}{((n-3)!)^2(2n-11)}.
$$

Вычисления констант $\Lambda_{n,3}$ и $\Lambda_{n,5}$, а также результаты работ \cite{Kalyab}, \cite{NazM} позволяют сформулировать следующую гипотезу.
\begin{hyp}
Функция $A^2_{n,k}(t)$ по переменной $t\in\left[-\frac14;0\right]$ имеет $\left\lceil\frac{k+1}{2}\right\rceil$ точек максимума.
Глобальный максимум величины $A^2_{n,k}(t)$  принимается в точке максимума, ближайшей к $-\frac14$. Для четных $k$ ближайшей такой точкой является сама точка $-\frac14$. В переменной $a\in [0;1]$ этому значению отвечает $a=\frac12$, а на отрезке $[-1;1]$  $a=0$ соответственно.
\end{hyp}

\section{Краевая задача, связанная с константами вложения.}\label{sec:BP}
В этом разделе мы установим связь констант вложения с некоторым классом спектральных задач и укажем способ быстрого пересчета точных значений констант вложения определенных на отрезках $[-1;1]$ и $[0;1]$ соответственно.

Зафиксируем на отрезке $[0;1]$ произвольную точку $a$ и рассмотрим следующую краевую задачу
\begin{gather}\label{eq:krzad1}
(-1)^ny^{(2n)}=\lambda (-1)^k y^{(k)}(a)\delta^{(k)}(x-a),\\\label{eq:krzad2}
y^{(j)}(0)=y^{(j)}(1)=0,\quad j=0,1,\ldots,n-1.
\end{gather}
Правую часть \eqref{eq:krzad1} можно  также представить в виде $\lambda\langle\delta^{(k)}(x-a),y\rangle \delta^{(k)}(x-a)$.

Операторной моделью для этой задачи служит пучок операторов $T_a:\; \Wo^{n}_2[0;1]\to \Wo^{-n}_2[0;1]$, квадратичная форма которого удовлетворяет соотношению
$$
\forall y\in W^{n}_2[0;1]\quad \langle T_a y,y\rangle:=\int_0^1|y^{(n)}(x)|^2\,dx-\lambda |y^{(k)}(a)|^2=0.
$$
Рассмотрим задачу о минимизации собственного значения задачи \eqref{eq:krzad1}--\eqref{eq:krzad2} по параметру $a$. Из определения пучка $T_a$ следует, что это наименьшее собственное значение соответствует пучку $T_{1/2}$
и в точности равно $(A^2_{n,k})^{-1}$.

Это наблюдение позволяет легко установить связь между константами вложения $A^2_{n,k,[-1;1]}$, определенными на отрезке $[-1;1]$ и $A^2_{n,k,[0;1]}$ --- константами вложения, определенными на отрезке $[0;1]$.

Действительно, сделаем замену переменных $s=2x-1$, осуществляющую преобразование отрезка $[0;1]$ на отрезок $[-1;1]$. Обозначим $z(s):=y\left(\frac{s+1}{2}\right)=y(x)$, с учетом того, что $ds=2dx$, получим, что пучок $T_a$ перейдет в пучок $\tilde T_{2a-1}$, квадратичная форма которого на отрезке $[-1;1]$ удовлетворяет соотношению
$$
\forall z\in W^{n}_2[-1;1]\quad \langle \tilde T_{2a-1} z,z\rangle:=2^{2n-1}\int_{-1}^1|z^{(n)}(s)|^2\,ds-\lambda\cdot  2^{2k}|z^{(k)}(2a-1)|^2=0.
$$

Следовательно, если $\lambda$ есть собственное значение пучка $T_a$ на отрезке $[0;1]$ с собственной функцией $y$, то $\frac{\lambda}{2^{2n-2k-1}}$ есть собственное значение пучка $\tilde T_{2a-1}$ на отрезке $[-1;1]$ с собственной функцией $z$. Отсюда следует, что
$$
A^2_{n,k,[-1;1]}(a)=2^{2n-2k-1}A^2_{n,k,[0;1]}(2a-1).
$$

Частности то же соотношение верно и для констант вложения:
$$
\Lambda^2_{n,k,[-1;1]}=2^{2n-2k-1}\Lambda^2_{n,k,[0;1]}.
$$


\begin{thebibliography}{99}
\bibitem{Tikh} В.\,М.\,Тихомиров, \textit{Некоторые вопросы теории приближений.}//М.:\;Изд-во МГУ, 1976.

\bibitem{MagTikh} Г.\,Г.\,Магарил--Ильяев, В.\,М.\,Тихомиров, \textit{Выпуклый анализ и его приложения.}//М.: Эдиториал УРСС, 2000.

\bibitem{Kalyab} Г.\,А.\,Калябин, \textit{Точные оценки для производных функций из классов Соболева $\Wo^{r}_2(-1;1)$}//Труды МИАН, 2010, Т.269,  143--149.

\bibitem{NazM} Е.\,В.\,Мукосеева, А.\,И.\,Назаров, \textit{О симметрии эктремали в некоторых теоремах вложения}//
Зап. научн. сем. ПОМИ, 2014, т.425, 35--45.

\bibitem{BotWid} B\"{o}ttcher A., Widom H., From Toeplitz eigenvalues through Green’s kernels to higher-order Wirtinger–Sobolev inequalities // The extended field of operator theory. Basel: Birkhauser, 2007, 73-–87. (Oper. Theory Adv. and Appl.; V. 171).

\bibitem{HolSh} K.\,В.\,Холшевников, В.\,Ш.\,Шайдулин \textit{О свойствах интегралов от многочлена Лежандра}//Вестник СПбГУ, Сер.1, 2014, т.1 (59), вып. 1, 55--67.

\end{thebibliography}
\end{document}